\documentclass{amsart}

\usepackage{xypic, amssymb, amsfonts, amsbsy, amsthm, amsmath, amscd, epsfig, latexsym, stmaryrd, epic, eepic, eucal}
\usepackage[english]{babel}
\DeclareMathAlphabet{\mathpzc}{OT1}{pzc}{m}{it}

\newtheorem*{Thm}{Theorem}
\newtheorem{theorem}{\bf Theorem}
\newtheorem{lemma}[theorem]{\bf Lemma}
\newtheorem{propos}[theorem]{\bf Proposition}

\newtheorem{claim}[theorem]{\bf Claim}

\theoremstyle{definition}
\newtheorem{defi}[theorem]{\bf Definition}
\newtheorem{oss}[theorem]{\bf Remark}

\newtheorem{notet}[theorem]{\bf Notation}
\newtheorem*{conj}{Conjecture}
\newtheorem*{thma}{Theorem 1}
\newtheorem*{thmb}{Theorem 2}

\newcommand{\lcm}{\operatorname{lcm}}

\newcommand{\N}{\mathbb N}

\newcommand{\A}{\mathbb A}
\newcommand{\Z}{\mathbb Z}

\newcommand{\Hht}{\operatorname{ht}}

\newcommand{\image}{\operatorname{Im}}

\newcommand{\kernel}{\operatorname{Ker}}

\begin{document}

\title{On the Periodicity of the First Betti Number of the Semigroup Rings under Translations}

\author{Adriano Marzullo}
\address{Department of Mathematics, Becker College, Worcester
Massachusets, 01609}
\email{adriano.marzullo@becker.edu}

\date{September, 14 2012}

\keywords{Semigroup Rings and Betti Numbers }

\begin{abstract}

 Let $k$ be a field of characteristic zero.
 Given an ordered $3$-tuple of positive integers $\mathbf{a}=(a,b,c)$ and for $j\in\N_{\geq 1}$, a family of sequences $\mathbf{\underline{a}}_j = \left(j,\;a+j,\;a+b+j,\;a+b+c+j\right)$, we consider the collection of monomial curves in $\A^{4}$ associated with $\mathbf{\underline{a}}_j$.
  The Betti numbers of the Semigroup rings collection associated with $\mathbf{\underline{a}}_j$ are conjectured to be eventually periodic with period $a+b+c$ by  Herzog and Srinivasan. Let $p\in\N$, in this paper, we prove that for $\mathbf{a}=(p(b+c),b,c)$ or $\mathbf{a}=(a,b,p(a+b))$ in the collection of defining ideals associated with $\mathbf{\underline{a}}_j$, for large $j$ the ideals are complete intersections if and only if $(a+b+c)| j$. Moreover, the complete intersections are periodic with the conjectured period.
\footnotemark
\end{abstract}
\maketitle
\vspace{2mm}
\footnote{This is the author's Ph.D. thesis, under the direction of Professor Hema Srinivasan. The author sincerely thanks Professor Srinivasan for suggesting the problem and for many helpful discussions concerning the material in this paper.}

\section*{Introduction}
\vspace{3mm}

  Let $k$ be a field of characteristic zero. Let $\N$ denote the nonnegative integers. Consider the finite sequence
$\;\mathbf{a}=\left(\textsf{a}_1,\textsf{a}_2,\ldots,\textsf{a}_{n}\right)\in\N^{n}$. The numerical semigroup generated by $\;\mathbf{a}$ is given by $\; S_{\mathbf{a}}=\textsf{a}_1\N+\textsf{a}_2\N+\cdots +\textsf{a}_{n}\N=\left\langle \textsf{a}_1, \textsf{a}_2, \ldots \textsf{a}_{n} \right\rangle$.

Consider the monomial curve $ \Gamma_{\mathbf{a}}=\left\{\left(t^{\textsf{a}_1},t^{\textsf{a}_2},\ldots, t^{\textsf{a}_{n}}\right)\in \A^{n}_k \;: t\in k\right\}$, then the Monomial Prime Ideal or the defining Ideal of $ \Gamma_{\mathbf{a}}$ is given by 
$$ P_{\mathbf{a}}=P(\textsf{a}_1,\textsf{a}_2,\ldots,\textsf{a}_{n})=\kernel{\varphi}$$
where $\varphi$ is the homomorphism of $k$-algebras:
$$\varphi\;:\;R=k\left[
x_1,x_2,\ldots,x_{n}\right]\longrightarrow k[t]\;\;\;x_i\longmapsto t^{\textsf{a}_i},\;\;\;\;\forall\; i=1,2,\ldots,n.$$
 The image of $\varphi$, denoted by $k[S_\mathbf{a}]$ or $k[\Gamma_{\mathbf{a}}]$, is the \emph{semigroup ring} associated with the numerical semigroup $S_\mathbf{a}$.

 Closely related to the map $\varphi$ is the homomorphism
\begin{equation}\label{secondamappa}
 \psi : \Z^{n}\longrightarrow \Z,\;\;\;\psi(e_i)=\textsf{a}_i, \;\;i\in{1,2,\ldots,n}
\end{equation}

 The image of $\psi$, $\image \left(\psi\right)=S_{\mathbf{a}}=\left\langle \textsf{a}_1,\textsf{a}_2,\ldots,\textsf{a}_{n}\right\rangle$. The kernel of $\psi$ plays an important role so we denote it by $K=\kernel(\psi)$.
As a consequence of the link between the map $\phi$ and $\psi$, we see that a binomial 

\begin{equation}\label{fond}
g=x^{\alpha}-x^{\beta}=x_1^{\alpha_1}\cdot x_2^{\alpha_2}\cdots x_n^{\alpha_n}-x_1^{\beta_1}\cdot x_2^{\beta_2}\cdots x_n^{\beta_n}\in P_{\mathbf{a}} \;\;\Leftrightarrow$$ $$\;\;\alpha-\beta=(\alpha_1-\beta_1,\alpha_2-\beta_2,\ldots,\alpha_{n}-\beta_{n})\in K.
\end{equation}
\vspace{2mm}

Let us recall the following general definitions:

\begin{defi}\cite{Alcil}\label{papino}
A binomial 
$ f_{i}=x_{i}^{\alpha_i}-\prod_{j\neq i} x_{j}^{\alpha_{ij}}\in P_{\mathbf{a}}$
is called a \emph{critical binomial} with respect to $x_i$ if $\alpha_i$ is the least positive integer such that
$
 \alpha_i\cdot \textsf{a}_i\in \left\langle \textsf{a}_1,\;\textsf{a}_2,\ldots,\hat{\textsf{a}}_i,\ldots \textsf{a}_{n}\right\rangle.
$

\end{defi}
\vspace{1mm}

\begin{defi}\cite{Alcil}
A set $ \left\{ f_1,f_2,\ldots, f_{n} \right\}$ is called a \emph{full set} of critical binomials 
if $f_i$ is a critical binomial with respect to $x_i$ for all $i\in {1,2,\ldots n}$.
\end{defi}
\vspace{2mm}
 It is proved in \cite{AdrianoMarzullo} that there is always a minimal generators set for $P_{\mathbf{a}}$ containing the full set of critical 
binomials:

\begin{propos}\label{Worcester}
 Let $P_{\mathbf{a}}$ be the defining ideal of a monomial curve in $\A_{k}^{n}$ and $\left\{f_1,f_2,\ldots,f_{n}\right\}$ be a critical set of binomials. Then $\left\{f_1,f_2,\ldots,f_{n}\right\}$ is part of a minimal system  of generators of $P_{\mathbf{a}}$. \qed
 \end{propos}
 \vspace{2mm}
 
 By Corollary $2.5.7$ in \cite{MonoAlg}, recall the following definition:
 
 \begin{defi}
 Let $R=\bigoplus_{i=1}^{\infty} R_i$ be a polynomial ring of dimension $n$ over a field $k$ and let M be a $\N$-graded R-module, then the minimal free resolution of M is given by:
   $$
   0\longrightarrow \bigoplus\limits_{i=1}^{b_{g}} R(-d_{gi})\stackrel{\varphi_{g}}{\longrightarrow}\cdots\longrightarrow \bigoplus\limits_{i=1}^{b_{k}} R(-d_{ki})\stackrel{\varphi_{k}}{\longrightarrow}\cdots
   $$
   $$
     \cdots \longrightarrow \bigoplus\limits_{i=1}^{b_{1}} R(-d_{1i})\stackrel{\varphi_{1}}{\longrightarrow}\bigoplus\limits_{i=1}^{b_{0}} R(-d_{0i})\stackrel{\varphi_{0}}{\longrightarrow} M\longrightarrow 0
   $$
   \vspace{1mm}
   
   The integers $b_{0},\ldots, b_{g}$, the ranks of the graded modules,  are called the Betti numbers of M, the $d_{ji}$ are the twists and they indicate a shift in the graduation. In particular, $b_0$ is the minimal number of generators of M.
\end{defi}

    Note that for any permutation $\mathbf{a^{'}}=(\textsf{a}_1^{'},\textsf{a}_2^{'},\ldots, \textsf{a}_{n}^{'})$ of 
 $\;\mathbf{a}=\left(\textsf{a}_1,\textsf{a}_2,\ldots,\textsf{a}_{n}\right)$ we get that:  
 
 $$S_{\mathbf{a}}=S_{\mathbf{a^{'}}},\;\;\;P_{\mathbf{a}}\cong P_{\mathbf{a^{'}}},\;\;\;\Gamma_{\mathbf{a}}\cong \Gamma_{\mathbf{a^{'}}}$$
 \vspace{1mm}
 
  so we can always assume that:$\;\textsf{a}_1\leq \textsf{a}_2 \leq \cdots \leq \textsf{a}_{n}$.
  
  \vspace{2mm}
  
   Let $d=\gcd(\textsf{a}_1,\textsf{a}_2,\ldots,\textsf{a}_{n})\;$ and consider the sequence given by:
   
  $$\; \mathbf{\frac{a}{d}}=\left(\frac{\textsf{a}_1}{d},\;\frac{\textsf{a}_2}{d},\;\ldots,\;\frac{\textsf{a}_{n}}{d}\right)\;\;\mbox{then}\;\; 
   P_{\mathbf{a}}= P_{\mathbf{\frac{a}{d}}}$$
   
\vspace{1mm}
   
Given $\;\mathbf{a}_j=\left\{\left(\textsf{a}_1+j,\textsf{a}_2+j,\ldots,\textsf{a}_{n}+j\right)\;|\;j\in \N\right\}$ a collection of sequences, let 
$P_{\mathbf{a}_j}, \varphi_{j}, \psi_{j}$  be the collection of defining ideals, homomorphisms of $k$-algebras and homomorphisms with kernels $K_{j}$ associated to $\mathbf{a}_j$.
\vspace{1mm}

\begin{notet}
 For any $j\in \N$, let $P_{\mathbf{a}_j}$ be the ideal defined above. In in this paper we will set $b_0=\mu\left(P_{\mathbf{a}_j}\right)$.
 \end{notet}
  \vspace{1mm}
  
  \begin{defi}\label{ci}
  For any $j\in \N$, let $P_{\mathbf{a}_j}$ be the ideal defined above, then $P_{\mathbf{a}_j}$ is a complete intersection ideal if  $\Hht\left(P_{\mathbf{a}_j}\right)=\mu\left(P_{\mathbf{a}_j}\right)$
 \end{defi}
 \vspace{1mm}

In general we will say that for every $j\in\N$, the sequence $\mathbf{a}_j$ defines a complete intersection if and only if the ideal $P_{\mathbf{a}_j}$ is a complete intersection ideal.

\begin{conj}\cite{Conjecture} $\mathbf{[\mbox{\textbf{J. Herzog, H. Srinivasan}}]}$

\begin{enumerate}
\item The Betti Numbers of $P_{\mathbf{a}_j}$ are eventually periodic in $j$;

\item $\left\{\mu\left(P_{\mathbf{a}_j}\right)\;|\;j\in \N\right\}$ is eventually periodic in $j$.

 \item $\left\{\mu\left(P_{\mathbf{a}_j}\right)\;|\;j\in \N\right\}$ is bounded.
\end{enumerate}
\end{conj}

 The above conjecture is true for $n=3$, see \cite{HerzogSrinivasan}.  Moreover in this case, considering
the monomial curve associated to the sequence $\mathbf{a}=\left(q,q+a,q+a+b\right) $, we have the following Herzog-Srinivasan characterization for 
complete intersection ideals:
 
 \begin{lemma}\cite{HerzogSrinivasan}\label{HerzogSrinivasan}
If $q\geq \max\left\{ab+b^2,\;ab+a^2\right\}$ then $\mathbf{a}=\left(q,q+a,q+a+b\right) $ defines a complete intersection ideal
if and only if there exist $\gcd\left(q,\;a+b\right)=x\neq 1$ and two non-negative integers $\alpha,\; \beta$ such that 
$$ x(q+a)=\alpha q+\beta\left(q+a+b\right).$$
Moreover, in this case, $\alpha+\beta=x$ and $\gcd\left(a,b\right)=t$ with $a+b=tx$.

 In particular, if $a$ and $\;b$ are relatively prime, $\left(q,q+a,q+a+b\right) $ is a complete intersection if and only if $a+b$ divides $q$. \qed

\end{lemma} 
\vspace{5mm}

In the paper we prove a partial result for $n=4$:

\begin{Thm}\label{MAIN}
Let $\mathbf{\underline{a}}=\left(1,\;1+a,\;1+a+b,\;1+a+b+c\right)$, $c=p(a+b)$ or $a=p(b+c)$, $j\in \N_{\geq 1}$ then
\begin{enumerate}
\item $\mu\left(P_{\mathbf{\underline{a}_{j}}}\right)=3$, occurs eventually periodically with period $a+b+c$ starting with $j=a+b+c$;
\vspace{1mm}
\item $\mu\left(P_{\mathbf{\underline{a}}_{j}}\right)=4$ occurs eventually periodically with period $a+b+c$;
\vspace{1mm}

\item For $j\geq \left(a+b+c\right)^3, P_{\mathbf{\underline{a}}_j}$
is a complete intersection if and only if $$(a+b+c)|j.$$
\end{enumerate}
\end{Thm}
\vspace{7mm}

\section{Srinivasan's Semigroup Rings}
 \vspace{3mm}
 
  Let $\mathbf{a}=\left(\textsf{a}_1,\textsf{a}_2,\ldots,\textsf{a}_{m}\right)\in \N^{m}$.
 To any sequence $\mathbf{a}$ of length $m$ we can associate the corresponding sequence of length 
$m+1$ given by:

\begin{equation}\label{aaaaa}
\mathbf{\underline{a}}=\left(1,\;1+\textsf{a}_1,\;1+\textsf{a}_1+\textsf{a}_2,\ldots,1+\textsf{a}_1+\textsf{a}_2+
\cdots+\textsf{a}_{m}\right)\in \N^{m+1}
\end{equation}

 The semigroup ring associated with $(\ref{aaaaa})$ is
 
\begin{equation}
k[S_{\mathbf{\underline{a}}}]=k\left[t,\;t^{1+\textsf{a}_1},\;t^{1+\textsf{a}_1+\textsf{a}_2},\ldots,t^{1+\sum_{i=1}^{m}\textsf{a}_i} \right]
\end{equation}

  and the monomial prime ideal $P_{\mathbf{\underline{a}}}$ has height $m$.
  
 Consider $j\in \N$. To $(\ref{aaaaa})$, we can associate the collection of sequences given by:

$$
\mathbf{\underline{a}_{j}}=\left(1+j,\;1+\textsf{a}_1+j,\;1+\textsf{a}_1+\textsf{a}_2+j,\ldots,1+\textsf{a}_1+\textsf{a}_2+\cdots+\textsf{a}_{m}+j\right)
$$

that we redefine as:

\begin{equation}\label{coooo}
\mathbf{\underline{a}_{j}}=\left(j,\;\textsf{a}_1+j,\;\textsf{a}_1+\textsf{a}_2+j,\ldots,\textsf{a}_1+\textsf{a}_2+\cdots+\textsf{a}_{m}+j\right)
\end{equation}

for $j\in \N_{\geq 1}$, and the collection of semigroup rings given by:

\begin{equation}\label{insurance}
k[S_{\mathbf{\underline{a}}_j}]=k\left[t^{j},\;t^{\textsf{a}_1+j},\;t^{\textsf{a}_1+\textsf{a}_2+j},\ldots,t^{\sum_{i=1}^{m}\textsf{a}_i+j} \right]
\end{equation}

Then,   $\mathbf{a}$ generates  a class of semigroup rings by translation 
\begin{equation}\label{SrinivasanSemigroupRings}
 F_{\mathbf{a}}=\left\{k[S_{\mathbf{\underline{a}}_j}],\;|\; j\in \N_{\geq 1}\right\}
\end{equation}
which we call the Srinivasan's class of semigroup rings generated by $\mathbf{a}$. 
\vspace{2mm}

 The Herzog-Srinivasan Conjecture states that the first Betti number of the Srinivasan's  class of semigroup rings $F_{\mathbf{a}}$ is eventually periodic.

\vspace{5mm}

\section{Problem's Setting in $\A^{4}$}
\vspace{3mm}

   Let $R=k[x_1,x_2,x_3,x_4],\; n=4,\;m=3$ and  $\;a,b,c\in \Z^+$. Then:
    \begin{equation}\label{setting1}
 \mathbf{a}=\left(\textsf{a}_1,\;\textsf{a}_2,\;\textsf{a}_3\right)=\left(a,\;b,\;c\right)
 \end{equation}
  
  Moreover, we have that the Srinivasan's semigroup rings corresponding to $(\ref{setting1})$ is led by:
  \begin{equation}\label{setting2}
 \mathbf{\underline{a}}=\left(1,\;1+a,\;1+a+b,\;1+a+b+c\right)\in \N^{4}
 \end{equation}
  and the collection of sequences associated to  $(\ref{setting2})$ is:
 \begin{equation}\label{setting3}
  \mathbf{\underline{a}}_j = \left(j,\;a+j,\;a+b+j,\;a+b+c+j\right)\in\N^{4}
 \end{equation} 
 
 where $j\in\N_{\geq 1}$. Here the conjectured period of the first Betti number of the Srinivasan's semigroup rings is $\;a+b+c$.
\vspace{1mm}

Let $\mathcal{M}\subset P_{\mathbf{a}_{j}}$ be a set of binomials with one coefficient $1$, the other $-1$ and the two terms 
 relatively prime to each other:
 
 $$ \mathcal{M}=\left\{g_1,\;g_2,\;g_3,\;g_4,\;g_5,\;g_6,\;g_7\right\}$$

where:

$$g_1=x_1^{\alpha_1}x_2^{\alpha_2}-x_3^{\alpha_3}x_4^{\alpha_4},\;\;g_2=x_1^{\alpha_1}x_3^{\alpha_3}-x_2^{\alpha_2}x_4^{\alpha_4};$$
$$g_3=x_1^{\alpha_1}x_4^{\alpha_4}-x_2^{\alpha_2}x_3^{\alpha_3},\;\;g_4=x_1^{\alpha_1}-x_2^{\alpha_2}x_3^{\alpha_3}x_4^{\alpha_4};$$
$$g_5=x_2^{\alpha_2}-x_1^{\alpha_1}x_3^{\alpha_3}x_4^{\alpha_4},\;\;g_6=x_3^{\alpha_3}-x_1^{\alpha_1}x_2^{\alpha_2}x_4^{\alpha_4};$$
$$g_7=x_4^{\alpha_4}-x_1^{\alpha_1}x_2^{\alpha_2}x_3^{\alpha_3};\;\;\;\;\;\;\;\;\;\;\;\;\;\;\;\;\;\;\;\;\;\;\;\;\;\;\;\;\;\;\;\;\;\;\;\;\;$$
\vspace{2mm}

\begin{lemma}\label{general}
Let $j\in \N_{\geq 1}$ and $(\alpha_1,\alpha_2,\alpha_3,\alpha_4)$ be the exponent of a binomial in $\mathcal{M}$. Then $(\alpha_1,\alpha_2,\alpha_3,\alpha_4)\in K_j$
if and only if the $\alpha_i$'s satisfy the following equation:
\begin{equation}\label{generalform}
j\left(\sum_{i=1}^{4} \alpha_i\right)+a\alpha_2+(a+b)\alpha_3+\alpha_4(a+b+c)=0
\end{equation}
\end{lemma}
\begin{proof}
$(\alpha_1,\alpha_2,\alpha_3,\alpha_4)=\alpha_1\cdot e_1+\alpha_2 \cdot e_2+\alpha_3 \cdot e_3 +\alpha_4 \cdot e_4\in K_j\subset \Z^4$  if and only if
$$ \psi_j\left(\alpha_1,\alpha_2,\alpha_3,\alpha_4\right)=\alpha_1\psi_j\left(e_1\right)+\alpha_2\psi_j\left(e_2\right)+\alpha_3\psi_j\left(e_3\right) +\alpha_4\psi_j\left(e_4\right)=0 $$
that is
$$ \alpha_1\left(j\right)+\alpha_2\left(a+j\right)+\alpha_3\left(a+b+j\right) +\alpha_4\left(a+b+c+j\right)=0\;\;\;\Leftrightarrow$$
$$j\left(\alpha_1+\alpha_2+\alpha_3+\alpha_4\right)+a\alpha_2+(a+b)\alpha_3+\alpha_4(a+b+c)=0$$
where the $\alpha_i$ for $i=1,2,3,4$ are considered according their signs. 
\end{proof}

\vspace{2mm}

\section{ Main Theorems}
\vspace{3mm}

\begin{thma}
 Let$\;n, t, p\in \N$, $1\leq t\leq p$ and $j\in \N_{\geq 1}$ . Then
    \begin{enumerate}
	           \item [(i)] If $c=p(a+b)$ or $a=p(b+c)$ 
	           and $\;j=(a+b+c)n\;$ then $\;\mu\left(P_{\mathbf{\underline{a}}_{j}}\right)=3;$

	           \item [(ii)]   If $\;c=p(a+b)$ and $\;j=(a+b+c)n+(a+b)t\;$ then $\;\mu\left(P_{\mathbf{\underline{a}}_{j}}\right)=4;$

	           \item [(iii)]If $\;a=p(b+c)$ and $\;j=(a+b+c)n+(b+c)t\;$ then $\;\mu\left(P_{\mathbf{\underline{a}}_{j}}\right)=4$.
	            
             \end{enumerate}

\end{thma}

\begin{proof}

\textbf{(i)} $\;$ We prove the theorem for $c=p(a+b)$  and $j=(a+b+c)n$. The proof for the case  $a=p(b+c)$ is similar and it was done in \cite{AdrianoMarzullo}.
 Under these hypotheses, rewrite $(\ref{generalform})$, as
\begin{equation}\label{jcpab}
(p+1)(a+b)n\left(\alpha_1+\alpha_2+\alpha_3+\alpha_4\right)+a\alpha_2+(a+b)\alpha_3+(p+1)(a+b)\alpha_4=0
\end{equation}

 Suppose that $\gcd(a,b)=d\neq 1$, then $a=da^{'}$ and $b=db^{'}$, where $\gcd(a^{'},b^{'})=1$. By $(\ref{jcpab})$ we have:
$$d\left[(p+1)(a^{'}+b^{'})n\left(\alpha_1+\alpha_2+\alpha_3+\alpha_4\right)+a^{'}\alpha_2+(a^{'}+b^{'})\alpha_3+(p+1)(a^{'}+b^{'})\alpha_4\right]=0.$$
 So, without loss of generality assume that $\gcd(a,b)=1$. Let
 
$$ I=\left\langle x_1^{n+1}-x_4^{n},\; x_3^{p+1}-x_1^{p}x_4,\;x_2^{a+b}-x_1^{b}x_3^{a} \right\rangle$$
 It is a computation to show that $I\subset P_{\mathbf{\underline{a}_{j}}}$. We need to prove the other inclusion: all the binomials in $P_{\mathbf{\underline{a}_{j}}}$ can be written as a combination of the generators of $\;I$.\\

We divide the proof in two steps:
\medskip

 \underline{Case $1$:} $\alpha_1+\alpha_2+\alpha_3+\alpha_4=0$. $\;(\alpha_1,\alpha_2,\alpha_3,\alpha_4)\in K_j\subset\Z^{4}$ is a solution of $(\ref{generalform})$ if and only if satisfies the following linear system:

\begin{equation}\label{tat}
\left\{\begin{array}{rl}
\alpha_1+\alpha_2+\alpha_3+\alpha_4=0\;\;\;\;\;\;\;\;\;\;\;\;\;\;\;\;\;\;\;\;\;\;\;\;\; & \;\;\;\;\\ 
a\alpha_2+(a+b)\alpha_3+(p+1)(a+b)\alpha_4=0&\;\;\;\;\\
\end{array}\right.
\end{equation}

It was completely proved in \cite{AdrianoMarzullo} that:

\begin{enumerate}
\item Two indipendent solutions of $(\ref{tat})$ are: $$(\alpha_1,\alpha_2,\alpha_3,\alpha_4)=(-p,0,p+1,-1)\;\;\mbox{and}\;\; (\alpha_1,\alpha_2,\alpha_3,\alpha_4)=(-b,a+b,-a,0)$$
\item Let $c_1,c_2\in\Z.$ Every $(\alpha_1,\alpha_2,\alpha_3,\alpha_4)$ solution of $(\ref{tat})$ can be expressed as $$(\alpha_1,\alpha_2,\alpha_3,\alpha_4)=c_1(-p,0,p+1,-1)+c_2(-b,a+b,-a,0)$$
\item Let $g_1$ and $g_2$ be the binomials corresponding to the two indipendent solutions of $(\ref{tat})$. Then every binomial in $ P_{\mathbf{\underline{a}_{j}}}$ satisfying $\;\alpha_1+\alpha_2+\alpha_3+\alpha_4=0$ is inside the ideal generted by $g_1$ and $g_2$. 
\end{enumerate}
\medskip

 \underline{Case $2$:} $\alpha_1+\alpha_2+\alpha_3+\alpha_4\neq 0$. Without loss of generality we can consider $\alpha_1+\alpha_2+\alpha_3+\alpha_4$ to be positive. Since $c=p(a+b)$ and $\;j=(a+b+c)n-1$, we already saw that $(\ref{generalform})$ becomes $(\ref{jcpab})$, that is:
 
 $$(p+1)(a+b)n\left(\alpha_1+\alpha_2+\alpha_3+\alpha_4\right)+a\alpha_2+(a+b)\alpha_3+(p+1)(a+b)\alpha_4=0$$

 Now, since $\gcd(a,b)=1$, by $(\ref{jcpab})$ we conclude that $\;a+b\;|\;\alpha_2$. Let's denote with
$$ I_2=\left\langle x_3^{p+1}-x_1^{p}x_4, x_1^{n+1}-x_4^{n} \right\rangle \subset I.$$
	
	\begin{itemize}
	
	\item Let $\alpha_2=0$. It was proved in \cite{AdrianoMarzullo} that every binomial $g_i\in \mathcal{M}$ is in $P_{\mathbf{\underline{a}_{j}}}$ for $i\in \left\{1,2,3,4,6,7\right\}$. To give an idea on how we did this, let us show that $g_1\in  P_{\mathbf{\underline{a}_{j}}}$ .
\vspace{2mm}

$\;g_1=x_1^{\alpha_1}-x_3^{\alpha_3}x_4^{\alpha_4}\in P_{\mathbf{\underline{a}_{j}}} \;\Leftrightarrow\;(\alpha_1,\alpha_2,\alpha_3,\alpha_4)=(k_1,0,-k_3,-k_4)\in K_j$
now, by Lemma \ref{general} this is true if and only if $\;(k_1,0,-k_3,-k_4)$ satisfies $\;(\ref{generalform})$
that is 
\begin{equation}\label{newgeneralform1}
\;\;k_1=k_3+k_4+\frac{k_3}{(p+1)n}+\frac{k_4}{n}
\end{equation}

 on the other hand:

\begin{equation}\label{mii1}
0<k_1-k_3-k_4=\frac{k_3}{(p+1)n}+\frac{k_4}{n}=\frac{\frac{k_3}{p+1}+k_4}{n}\in\Z^{+}
\end{equation}
\vspace{1mm}

 If 
$n>\frac{k_3}{p+1}+k_4\;$ then $\;\frac{\frac{k_3}{p+1}+k_4}{n}<1$
 and this is a contradiction. So there is no solution for $\;n>\frac{k_3}{p+1}+k_4\;$.
If $n\leq\frac{k_3}{p+1}+k_4$, then

\begin{equation}\label{mom1}
\frac{k_3}{p+1}+k_4=Hn\;\;\;\mbox{where}\;\;H\in \N.
\end{equation}

 and by $\left(\ref{newgeneralform1}\right)$, we get that
\begin{equation}\label{pop1}
k_1=k_3+k_4+H.
\end{equation}

  \begin{enumerate}
 \item If $\alpha_3=0\;(\alpha_2=0)$, then by $(\ref{mom1})$ and by $(\ref{pop1})$ 
 
$$ g_1=x_1^{H(n+1)}-x_4^{Hn}\in\left\langle x_1^{(n+1)}-x_4^{n} \right\rangle\subset I.$$

\vspace{1mm}
  \item If $\alpha_3\neq 0\;(\alpha_2=0)$, then by $(\ref{mom1})$, we get $\;\frac{k_3}{p+1}\in\Z^{+}$, that is $\;(p+1)\;|\;k_3$. So
write

$$\left(k_1,0,-k_3,-k_4\right)+\left(-\frac{k_3 p}{p+1},0,\frac{(p+1)k_3}{p+1},-\frac{k_3}{p+1}\right)= \left(k_1-\frac{k_3 p}{p+1},0,0,-k_4-\frac{k_3}{p+1}\right).$$

Since 

\begin{equation}\label{pino1}
x_3^{\frac{(p+1)k_3}{p+1}}-x_1^{\frac{k_3 p}{p+1}}x_4^{\frac{k_3}{p+1}}\in \left\langle x_3^{p+1}-x_1^{p}x_4\right\rangle.
\end{equation}

we need to show that 
\begin{equation}\label{psss}
\;\;\;x_1^{k_1-\frac{k_3 p}{p+1}}-x_4^{k_4+\frac{k_3}{p+1}}\in\left\langle x_1^{n+1}-x_4^{n}\right\rangle.
\end{equation}

To prove $(\ref{psss})$ rewrite  $\;(\ref{newgeneralform1})$ as
$$ (p+1)n\left(k_1-k_3-k_4\right)=(p+1)k_4+k_3\;\;\Rightarrow\;\;$$
$$(p+1)nk_1-(p+1)n k_3=(p+1)k_4 +k_3+(p+1)nk_4\;\;\Rightarrow\;\;$$
$$n(p+1)\left(k_1-\frac{k_3 p}{p+1}\right)=nk_3+k_3+(p+1)k_4(n+1)\;\;\Rightarrow\;\; $$

\begin{equation}\label{son1}
n\left(k_1-\frac{pk_3}{p+1}\right)=\left(n+1\right)\left(k_4+\frac{k_3}{p+1}\right).
\end{equation}

 Consider the following map:$$\xi:K[x_1,x_4]\rightarrow K[t]\;\;\;\mbox{where}\;\;\;x_1\mapsto t^{n},\;\;x_4\mapsto t^{n+1}.$$

 Of course, $\;\ker{\xi}= \left\langle x_1^{n+1}-x_4^{n} \right\rangle$ and by $\;\left(\ref{son1}\right)$ 
$$x_1^{k_1-\frac{k_3 p}{p+1}}-x_4^{k_4+\frac{k_3}{p+1}}\in\left\langle x_1^{n+1}-x_4^{n}\right\rangle.$$
and this proves $(\ref{psss})$.

Finally, by $(\ref{pino1})$ and $(\ref{psss})$, considering the map $\pi :R \rightarrow \frac{R}{I_2}$,
it is a computation to show that $g_1=x_1^{k_1}-x_3^{k_3}x_4^{k_4}$ is zero in $\;\frac{R}{I_2}$. This completes the proof for $\alpha_2=0$.
\end{enumerate}

\vspace{2mm}

\item Let $\alpha_2\neq 0$. In \cite{AdrianoMarzullo} we proved that every binomial $g_i\in \mathcal{M}$ is in $P_{\mathbf{\underline{a}_{j}}}$ for $i\in \left\{1,2,3,4,5,6,7\right\}$.
Here, let us prove that $g_2\in  P_{\mathbf{\underline{a}_{j}}}$.
\vspace{2mm}
 
$g_2=x_1^{\alpha_1}x_3^{\alpha_3}-x_2^{\alpha_2}x_4^{\alpha_4}\in  P_{\mathbf{a_{j}}}\;\Leftrightarrow\;(\alpha_1,\alpha_2,\alpha_3,\alpha_4)=(k_1,-k_2,k_3,-k_4)\in K_j.$
where $\;k_1-k_2+k_3-k_4>0$. Since $(a+b)| \alpha_2$, we can consider
$$ \left(k_1,-k_2,k_3,-k_4\right)+\left(-\frac{bk_2}{a+b},\;\frac{\left(a+b\right)k_2}{\left(a+b\right)},\;-\frac{ak_2}{a+b},\;0\right)=$$
$$=\left(k_1-\frac{bk_2}{a+b},\;0,\;k_3-\frac{ak_2}{a+b},\;-k_4\right).
$$
Let
$$\left(k_1^{'},0,k_3^{'},k_4^{'}\right)=\left(k_1-\frac{bk_2}{a+b},\;0,\;k_3-\frac{ak_2}{a+b},\;-k_4\right). $$
Now, it is hard to predict the signs of $\;k_{1}^{'}$ and $\;k_{3}^{'}$ because they could be either positive or negative. On the other hand, we can certainly 
say that they cannot be both negative $($actually in this case we do not have a binomial anymore$)$. Then we have three possibilities, and in all of them the binomial associated to $\;\left(k_1^{'},0,k_3^{'},k_4^{'}\right)$ is in $\;I$ for what we have done above. Moreover, it is just a computation to prove that

 \begin{equation}\label{a4}
 x_2^{\frac{\left(a+b\right)k_2}{\left(a+b\right)}}-x_1^{\frac{bk_2}{a+b}}x_3^{\frac{ak_2}{a+b}}\in 
 \left\langle   x_2^{a+b}-x_1^{b}x_3^{a}\right\rangle
 \end{equation}
\vspace{2mm}
 
\begin{claim}
$\;\;\;x_1^{k_1}x_3^{k_2}-x_2^{k_3}x_4^{k_4}\in I$
\end{claim}

 \begin{proof}
 Consider the map $\;\pi: R\rightarrow R/I$.
 \vspace{1mm}
 
\begin{enumerate} 
\item If $\;k_1^{'}<0$ and $\;k_3^{'}>0$ , then as mentioned above

\begin{equation}\label{b1}
 x_3^{k_3-\frac{ak_2}{a+b}}-x_1^{\frac{bk_2}{a+b}-k_1}x_4^{k_4}\in I_2 \subset I
\end{equation}

By $\;(\ref{a4})$, we can write
$$
 g_2=x_1^{k_1}x_3^{k_3}-x_2^{k_2}x_4^{k_4}+ x_2^{k_2}-x_1^{\frac{bk_2}{a+b}}x_3^{\frac{ak_2}{a+b}}+C_3\left(x_2^{a+b}-x_1^{b}x_3^{a}\right)
$$
where $\;C_3\in R$.
In $\;\frac{R}{I}$, we  have that $\;x_2^{k_2}=x_1^{\frac{bk_2}{a+b}}x_3^{\frac{ak_2}{a+b}}$. So the last equation becomes:
$$g_2=x_1^{k_1}x_3^{k_3}-\left(x_1^{\frac{bk_2}{a+b}}x_3^{\frac{ak_2}{a+b}}\right)x_4^{k_4}+ C_3\left(x_2^{a+b}-x_1^{b}x_3^{a}\right)$$
then
$$g_2=x_1^{k_1}x_3^{\frac{ak_2}{a+b}}\left(x_3^{k_3-\frac{ak_2}{a+b}}-x_1^{\frac{bk_2}{a+b}-k_1}x_4^{k_4}\right)+ C_3\left(x_2^{a+b}-x_1^{b}x_3^{a}\right) $$
then by $\;(\ref{b1})$, we get
$$g_2=x_1^{k_1}x_3^{\frac{ak_2}{a+b}}\left[C_1\left(x_3^{p+1}-x_1^{p}x_4\right)+C_2\left(x_1^{n+1}-x_4^{n}\right)\right]+ C_3\left(x_2^{a+b}-x_1^{b}x_3^{a}\right)$$
where $\;C_1,C_2\in R$.
\vspace{2mm}

\item For $k_1^{'}>0$ and $ k_3^{'}<0$ or  $\;k_1^{'}>0$ and $ k_3^{'}>0$, similar proofs are done in \cite{AdrianoMarzullo}.
\end{enumerate}
\vspace{2mm}

This concludes the proof of (i).
\end{proof}
 \end{itemize}
\vspace{3mm}


\textbf{(ii)} $\;$ Let  $j=(a+b+c)n+(a+b)t,\;c=p(a+b),\;1\leq t\leq p\;$ where $\;n,t,p\in\Z^{+}$. Again assume that
 $\gcd(a,b)=1$. 
 \vspace{1mm}
 
 Recall that $ S_{\mathbf{\underline{a}}_j}=\left\langle j,\;a+j,\;a+b+j,\;a+b+c+j\right\rangle$
where
$$\;j=(p+1)(a+b)n+(a+b)t=(a+b)[(p+1)n+t].$$
 
 Let $L:=[(p+1)n+t]\in \N$, then rewrite:
 
 $$ S_{\mathbf{\underline{a}}_j}=\left\langle \textsf{a}_1,\textsf{a}_2,\textsf{a}_3,\textsf{a}_4\right\rangle=$$
  \begin{equation}\label{TAMPA}
 =\left\langle \;(a+b)L,\;(a+b)L+a,\;(a+b)(L+1),\;(a+b)(L+p+1)\right\rangle
 \end{equation}

 Let $$ I=\left\langle x_1^{n+t+1}-x_3^{t}x_4^{n}, \;x_2^{a+b}-x_1^{b}x_3^{a}, \;x_3^{p+1}-x_1^{p}x_4, \;x_4^{n+1}-x_1^{n+t-p+1}x_3^{p-t+1}\right\rangle$$

 It is an computation to show that $I\subset P_{\mathbf{\underline{a}}_j}.$ We need to show the other inclusion. 
 \medskip
 
  First, it was proved in \cite{AdrianoMarzullo} that:
  $$I=\left\langle f_1,\;f_2,\;f_3,\; f_4\right\rangle$$
   where $f_i$ is the critical binomial with respect $x_i, \;i=1,2,3,4$. 
   To give the reader an idea on how this was done, let us show that:
   
   \begin{lemma}\label{f3}
$\;f_3=x_3^{p+1}-x_1^{p}x_4$.
\end{lemma}
\begin{proof}
By Definition \ref{papino}, we need to show that $p+1$ is the minimal integer such that 
\begin{equation}\label{p+1}
(p+1)\textsf{a}_3=k_1\textsf{a}_1+k_2\textsf{a}_2+k_4\textsf{a}_4\;\;\;\;\;\;\;\;\;k_1,k_2,k_4\in\Z^{+}
\end{equation}
\vspace{1mm}

It is proved in \cite{AdrianoMarzullo} that:
\begin{claim}\label{k2=0}
If $\alpha$ is the minimal positive integer such that 
\begin{equation}\label{s3p}
\alpha \textsf{a}_3=k_1\textsf{a}_1+k_2\textsf{a}_2+k_4\textsf{a}_4
\end{equation}
where $\;\;k_1,k_2,k_4\in\Z^{+},\;\;$ then $\;k_2=0$.\qed
\end{claim}
\vspace{1mm}

Let $\;\alpha<(p+1)$ and suppose that for every $n\in\N$ there exist $k_1,k_4\in\Z^{+}$ such that
\begin{equation}\label{s3p1}
\alpha \textsf{a}_3=k_1\textsf{a}_1+k_4\textsf{a}_4
\end{equation}
\vspace{2mm}

It is proved in \cite{AdrianoMarzullo} that:

\begin{claim}
In $\;(\ref{s3p1})$, if $\;\alpha<(p+1)\;$ then $\;k_4\geq 1.$ \qed
\end{claim}

So far, we got that $(\ref{s3p1})$  holds when $k_4 \geq 1$. 
Now, $(\ref{s3p1})$ can be rewritten as:
\begin{equation}\label{s3p3}
\alpha(L+1)=k_1L+k_4(L+p+1)
\end{equation}
$(\ref{s3p3})\;$  implies:
\begin{enumerate}
	\item $\;L(\alpha-k_1-k_4)=k_4(p+1)-\alpha\;\;\;\Longrightarrow\;\;\;$
	\begin{equation}\label{s3p4}
	\alpha>k_1+k_4
	\end{equation}
since, $\;k_4>1\;$ and $\;\alpha<p+1$.
\vspace{2mm}

\item $\;\alpha[(p+1)n+t+1]=k_1[(p+1)n+t]+k_4[(p+1)n+t+(p+1)]\;$ that is:
$$(p+1)[(\alpha-k_1-k_4)n-k_4]=-[(\alpha-k_1-k_4)t]-\alpha <0$$
because of $(\ref{s3p4})$. Then we get that
$$ (\alpha-k_1-k_4)n-k_4<0\;\;\Longrightarrow\;\;k_4>n(\alpha-k_1-k_4)>n\;\;\Longrightarrow$$
\begin{equation}\label{s3p5}
\;\; k_4>n
\end{equation}
Finally, by $(\ref{s3p4})$ and $(\ref{s3p5})$, we get
$$ \alpha>k_1+k_4>n\;\;\;\Longrightarrow\;\;\;n<\alpha<p+1.$$
In conclusion, what we got is that $(\ref{s3p1})\;$ does not hold for $n\geq p+1$, and this is a contradiction.
\end{enumerate}
\end{proof}

   
Second, consider now the following isomorphism $\delta:\;R\longrightarrow R$ given by $\;x_1\mapsto x_1$, $x_2\mapsto x_4$, $x_3\mapsto x_3$, $x_4\mapsto x_2$.
Let $\mathcal{I}=\delta\left(I\right)=\left\langle \mathcal{F}_1, \mathcal{F}_2, \mathcal{F}_3, \mathcal{F}_4 \right\rangle$ where:

$$
\mathcal{F}_1=\delta\left(f_1\right)=x_1^{n+t+1}-x_3^{t}x_2^{n} $$
$$\mathcal{F}_4=\delta\left(f_2\right)=x_4^{a+b}-x_1^{b}x_3^{a}$$
$$\mathcal{F}_3=\delta\left(f_3\right)=x_3^{p+1}-x_1^{p}x_2$$
$$\mathcal{F}_2=\delta\left(f_4\right)=x_2^{n+1}-x_1^{n+t-p+1}x_3^{p-t+1}.$$
\vspace{1mm}

Since the binomials $\mathcal{F}_i$ are critical, by Proposition 3.4 in \cite{Alcil} 
we can conclude that $\mathcal{I}$ is prime and so is I.
\medskip

Finally, consider the ideal $J$ generated by:
$$
J=\left\langle x_1^{\textsf{a}_2}-x_2^{\textsf{a}_1},\;x_1^{\textsf{a}_3}-x_3^{\textsf{a}_1},\;x_1^{\textsf{a}_4}-x_4^{\textsf{a}_1},\;x_2^{\textsf{a}_3}-x_3^{\textsf{a}_2},\;x_2^{\textsf{a}_4}-x_4^{\textsf{a}_2},\;x_3^{n_4}
-x_4^{n_3}\right\rangle
$$
By Corollary 10.1.10 in \cite{MonoAlg} we know that $\sqrt{J}=P_{\mathbf{\underline{a}}_j}$. To prove that $I\supset J$, we proved in \cite{AdrianoMarzullo} that each generator of $J$ is inside $I$. Consider the quotient ring $R/I$. Then, for $i,j\in\left\{1,2,3,4\right\}$:
  
$$x_i^{\textsf{a}_j}-x_j^{\textsf{a}_i}\in I \Longleftrightarrow x_i^{\textsf{a}_j}-x_j^{\textsf{a}_i}=0 \;\;\mbox{in}\;\; R/I \Longleftrightarrow x_i^{\textsf{a}_j}=x_j^{\textsf{a}_i}\;\;\;\mbox{in}\;\; R/I.$$
\vspace{1mm}

 To have an idea of how this proof goes,  show that $x_1^{\textsf{a}_2}-x_2^{\textsf{a}_1}\in I$.
\medskip

 By $(\ref{TAMPA})$:  $\;\;\;x_1^{\textsf{a}_2}=x_2^{\textsf{a}_1}\;\;\Leftrightarrow\;\; x_1^{(a+b)\left[(p+1)n+t\right]+a}=x_2^{(a+b)[(p+1)n+t]}\;\;\Leftrightarrow$
	
	$x_1^{(a+b)\left[(p+1)n+t\right]+a}=(x_1^bx_3^a)^{[(p+1)n+t]}\;\;\Leftrightarrow\;\;x_1^{a\left[pn+n+t+1\right]}=x_3^{a[(p+1)n+t]}\;\;\Leftrightarrow$
	
	$x_1^{apn}\left(x_1^{n+t+1}\right)^a=\left(x_3^{(p+1)}\right)^{an}x_3^{at}\;\;\Leftrightarrow\;\; x_1^{apn}\left(x_3^{t}x_4^{n}\right)^a=\left(x_1^px_4\right)^{an}x_3^{at}.$
	\vspace{6mm}

 So far we proved that not only the ideal $I$ is prime but also that $I\supset J$. 
By Corollary 10.1.10 in \cite{MonoAlg}, we know that $\sqrt{J}=P_{\mathbf{\underline{a}}_j}$. On the other hand, by definition:
$$P_{\mathbf{\underline{a}}_j}=\sqrt{J}=\bigcap_{P\supset J} P \subset I $$

 and this concludes our proof.

\vspace{3mm}


 \textbf{(iii)} $\;$ Let $a=p(b+c),\;j=(a+b+c)n+(b+c)t,\;1\leq t\leq p\;$ where $\;n,t,p\in\Z^{+}$. Without loss of generality we can assume that 
 $\gcd(b,c)=1$. Recall that the arithmetic semigroup we are considering is given by 
$$ S_{\mathbf{a}_j}=\left\langle j,\;a+j,\;a+b+j,\;a+b+c+j\right\rangle $$
where
$$\;j=(p+1)(b+c)n+(b+c)t=(b+c)[(p+1)n+t]$$
 now let $L:=[(p+1)n+t]\in \N$, then rewrite:
  $$S_{\mathbf{a}_j}=\left\langle \textsf{a}_1,\textsf{a}_2,\textsf{a}_3,\textsf{a}_4\right\rangle= $$
    \begin{equation}\label{Americus}
     =\left\langle (b+c)L,(b+c)(L+p),(b+c)(L+p)+b,(b+c)(L+p+1)\right\rangle
 \end{equation}

Let $$ I= \left\langle x_1^{n+2}-x_2^{p+1-t}x_4^{n+t-p}, \;x_2^{p+1}-x_1x_4^{p}, \;
x_3^{b+c}-x_2^{c}x_4^{b},\; x_4^{n+t}-x_1^{n+1}x_2^{t} \right\rangle. $$
 It is an computation to show that $I\subset P_{\mathbf{\underline{a}}_j}.$ We need to show the other inclusion. Similarly to (ii), it was proven in \cite{AdrianoMarzullo} that:
\begin{itemize}
	\item $I=\left\langle f_1,\; f_2,\; f_3,\; f_4\right\rangle$;
	\vspace{2mm}
	
	\item $I$ is a prime ideal;
	\vspace{2mm}
	
	\item $I\supset J=\left\langle x_1^{\textsf{a}_2}-x_2^{\textsf{a}_1},\;x_1^{\textsf{a}_3}-x_3^{\textsf{a}_1},\;x_1^{\textsf{a}_4}-x_4^{\textsf{a}_1},\;x_2^{\textsf{a}_3}-x_3^{\textsf{a}_2},\;x_2^{\textsf{a}_4}-x_4^{\textsf{a}_2},\;x_3^{n_4}
-x_4^{n_3}\right\rangle$
\end{itemize}
\vspace{2mm}

By Corollary 10.1.10 in \cite{MonoAlg}, we know that $\sqrt{J}=P_{\mathbf{\underline{a}}_j}$. On the other hand, by definition:
$$P_{\mathbf{\underline{a}}_j}=\sqrt{J}=\bigcap_{P\supset J} P \subset I $$

 and this concludes our proof.

\end{proof}
\vspace{3mm}

 \begin{thmb}

Let $\mathbf{a}=\left(a,b,c\right)\in\N^{3}$.
Let $p,n,r\in \N$, $\;j=(a+b+c)n+r$ and $a=p(b+c)$ or $c=p(a+b)$. Let
 \begin{equation}\label{vapensiero}
\mathbf{\underline{a}}_{j} =\left(j,a+j,a+b+j, a+b+c+j\right)\in \N^{4}
\end{equation}

 Then for $j\geq \left(a+b+c\right)^3$, the Srinivasan's semigroups rings $k[S_{\mathbf{\underline{a}}_j}]$ associated with the collection of monomial prime ideals $P_{\mathbf{\underline{a}}_j}$
is a complete intersection if and only if $(a+b+c)|j$. In particular, in the family of the Srinivasan's semigroups rings the complete intersection appears eventually with period $a+b+c$.

\end{thmb}
 
 \begin{proof}
 If $r=0$ then $j=(a+b+c)n$, so we are under the ipotheses of Theorem $1. (1)$ and by the Definition \ref{ci}, the collection of defining ideal 
 $P_{\mathbf{\underline{a}}_{j}}$ associated is a family of complete intersection ideals. We need to prove the converse: if $P_{\mathbf{\underline{a}}_{j}}$ is a family of complete intersection ideals then $r=0$. 
 \medskip

 Let $\;a=p(b+c)$. In this case,  by $(\ref{vapensiero})$
the arithmetic semigroup associated with the shifted sequence 
is given by $S_{\mathbf{\underline{a}}_{j}}=\left\langle \textsf{a}_1,\;\textsf{a}_2,\;\textsf{a}_3,\;\textsf{a}_4\right\rangle$, where:
 \begin{align*}
  \textsf{a}_1=&j=(p+1)(b+c)n+r\\
  \textsf{a}_2=&j+a=(p+1)(b+c)n+p(b+c)+r\\
  \textsf{a}_3=&j+a+b=(p+1)(b+c)n+p(b+c)+b+r\\
  \textsf{a}_4=&j+a+b+c=(p+1)(b+c)n+(p+1)(b+c)+r
  \end{align*}
 
 \medskip
 
 To prove Theorem $2$ for $a=p(b+c)$, we will use the following Lemma proved in \cite{AdrianoMarzullo}
 \begin{lemma}\label{italiafuori3}
Let $\textsf{a}_1,\textsf{a}_2,\textsf{a}_3,\textsf{a}_4$ as above.  If $r^{'}>0$ and $n\geq (b+c)^2$, then there is 
a relation, $\;\gamma_1\textsf{a}_1=\gamma_3\textsf{a}_3+\gamma_4\textsf{a}_4$ with all $\gamma_1,\;\gamma_3$ and $\gamma_4$ positive numbers, and $\gamma_1\leq c(n+1)\leq (b+c)(n+1)$. In particular, in the critical binomial

$\;f_1=x_1^{\alpha_1}-x_2^{\alpha_{12}}x_3^{\alpha_{13}}x_4^{\alpha_{14}}$, we have that $\alpha_1<(n+1)(b+c)$ \qed
\end{lemma}

\vspace{2mm}

  By $(\ref{generalform})$, we have that $ x_2^{p+1}-x_1x_4^{p},\;\;x_3^{b+c}-x_2^{c}x_4^{b}\in P_{\mathbf{\underline{a}}_{j}}.$ Moreover it is a computation to prove that:
  
  \begin{enumerate}
  \item In $f_3=x_3^{\alpha_3}-x_1^{\alpha_{31}}x_3^{\alpha_{32}}x_4^{\alpha_{34}} $
   at least two of $\alpha_{31}, \alpha_{32}, \alpha_{34}$ are not zero. As a consequence of this, we get 
   $\; f_3\neq -f_i,\;i\in\left\{1,2,4\right\}$.
 
 \medskip

  \item In $f_2=x_2^{\alpha_2}-x_1^{\alpha_{21}}x_3^{\alpha_{23}}x_4^{\alpha_{24}}$
  then at least two of $\alpha_{21}, \alpha_{23}, \alpha_{24}$ are not zero. As a consequence of this,  we get 
  $\; f_2\neq -f_i, \; i\in\left\{1,3,4\right\}$.
 \end{enumerate}

  By Proposition \ref{Worcester}, we know that $\left\{f_1,f_2,f_3,f_4\right\}$, a full set of critical binomials in $R$, is part of a minimal set of generators of $P_{\mathbf{\underline{a}}_j}$. Now, if the monomial prime ideal $P_{\mathbf{a}_j}$ is a complete intersection ideal, by Theorem $3.1$ in \cite{Alcil}, we must have two of the critical binomials equal up to change of sign. So we have that:
 \begin{equation}\label{industry}
 f_1=-f_4
 \end{equation}  
 
  Now, if $P_{\mathbf{\underline{a}}_j}$ is a complete intersection ideal in $R$ then  by Lemma \ref{HerzogSrinivasan} either
 
\begin{enumerate}
	\item $\mathbf{a}=\left\langle \textsf{a}_1,\textsf{a}_2,\textsf{a}_4\right\rangle$ defines a complete intersection ideal in $k[x_1,x_2,x_4]$ with 
	$\left\{f_1=-f_4,\; f_2\right\}$ as set of minimal generators \textemdash $\;$  it is a computation to modify Lemma \ref{general} for $n=3$ and prove that $f_1, f_2\in P_{\mathbf{a}}$, then by Proposition \ref{Worcester}, we have $P_{\mathbf{a}}=\left\langle f_1,f_2\right\rangle$ \textemdash $\;$ or
	\vspace{1mm}

	\item $\mathbf{a}=\left\langle \textsf{a}_1,\textsf{a}_3,\textsf{a}_4\right\rangle$ defines a complete intersection in $k[x_1,x_3,x_4]$ with 
	$\left\{f_1=-f_4,\; f_3\right\}$ as set of minimal generators \textemdash $\;$  again, it is a computation to modify Lemma \ref{general} for $n=3$ and prove that $f_1, f_3\in P_{\mathbf{a}}$, then by Proposition \ref{Worcester}, we have $P_{\mathbf{a}}=\left\langle f_1,f_3\right\rangle$.
	
\end{enumerate}

\medskip

 Let us analyze these cases separately.
\medskip

\begin{enumerate}
\item  If $\left\langle \textsf{a}_1,\textsf{a}_2,\textsf{a}_4\right\rangle$ defines a complete intersection ideal then by Lemma \ref{HerzogSrinivasan}, 
we have:
 
\begin{equation}\label{italiafuori1}
\alpha_2=\gcd\left(\;(p+1)(b+c)n+r,\;p(b+c)+(b+c)\;\right)=\gcd\left(\;r,\;(p+1)(b+c)\;\right)
\end{equation}
\vspace{1mm}

Moreover in this case:
$$ t=\gcd\left(p(b+c),\;b+c\right)=b+c. $$

Finally, again by Lemma \ref{HerzogSrinivasan}:

\begin{equation}\label{italiafuori2}
\alpha_2=\frac{p(b+c)+(b+c)}{b+c}=p+1
\end{equation}

 By $(\ref{italiafuori2})$, we have not only that $\;f_2=x_2^{p+1}-x_1x_4^{p}\;$ but also $($by $(\ref{italiafuori1}))$:

$$\gcd\left(r,\;(p+1)(b+c)\right)=p+1\;\;\Rightarrow\;\; (p+1)| r.$$

 Then:
$$r=(p+1)r^{'}<(p+1)(b+c)\;\;\Rightarrow\;\;\frac{r}{p+1}=r^{'}<b+c.$$ 

 Now, rewriting $\textsf{a}_1, \textsf{a}_2, \textsf{a}_3 $ and $\textsf{a}_4$ as:

\begin{align*}
  \textsf{a}_1=&(p+1)\left[n(b+c)+r^{'}\right]\\
  \textsf{a}_2=&(p+1)\left[n(b+c)+r^{'}\right]+p(b+c)\\
  \textsf{a}_3=&(p+1)\left[n(b+c)+r^{'}\right]+p(b+c)+b\\
  \textsf{a}_4=&(p+1)\left[(n+1)(b+c)+r^{'}\right]
  \end{align*}

 by Lemma \ref{HerzogSrinivasan}, we have that:

\begin{equation}\label{nonsoniente}
f_1=x_1^{(n+1)(b+c)+r^{'}}-x_4^{n(b+c)+r^{'}}
\end{equation}

  and this is a contradiction because of the Lemma \ref{italiafuori3}
 

 
\item  If $\left\langle \textsf{a}_1,\textsf{a}_3,\textsf{a}_4\right\rangle$ defines a complete intersection ideal then by Lemma \ref{HerzogSrinivasan}, 
we have:

\begin{equation}\label{italiafuori1222}
\alpha_3=\gcd\left((p+1)(b+c)n+r,\;p(b+c)+(b+c)\right)=\gcd\left(r,\;(p+1)(b+c)\right)
\end{equation}

 Moreover in this case:
$$ t=\gcd\left(p(b+c)+b,\;c\right)=\gcd\left((p+1)b+pc,\;c\right)=\gcd\left((p+1)b,\;c\right)$$

 Let $q=\gcd\left(b,c\right)$ then:
 $$ t=q\cdot \gcd\left((p+1)b^{*},\;c^{*}\right)$$
 
  where $\gcd\left(b^{*}, c^{*}\right)=1$.

 So Lemma \ref{HerzogSrinivasan} implies:

$$ \alpha_3=\frac{(p+1)(b+c)}{q\cdot \gcd\left((p+1)b^{*},\;c^{*}\right)}\\
         =\frac{q\cdot (p+1)(b^{*}+c^{*})}{q\cdot \gcd\left((p+1)b^{*},\;c^{*}\right)}=\frac{(p+1)(b^{*}+c^{*})}{\gcd\left((p+1)b^{*},\;c^{*}\right)}$$

 that is:

$$\alpha_3= \frac{(p+1)(b^{*}+c^{*})}{\gcd\left((p+1)b^{*},\;c^{*}\right)}=\frac{(p+1)(b^{*}+c^{*})}{(p+1)b^{*}c^{*}}\cdot \lcm \left((p+1)b^{*},\;c^{*}\right)\;\;\Rightarrow$$

$$ \alpha_3=\frac{(b^{*}+c^{*})}{b^{*}c^{*}}\cdot \lcm \left((p+1)b^{*},\;c^{*}\right)\;\;\Rightarrow$$

$$\alpha_3=\frac{(b^{*}+c^{*})}{b^{*}c^{*}}\cdot \lcm \left((p+1)b^{*},\;c^{*}\right)\geq \frac{(b^{*}+c^{*})}{b^{*}c^{*}}\cdot(p+1)b^{*}c^{*}=(b^{*}+c^{*})(p+1)>b^{*}+c^{*} \;\;\Rightarrow$$

$$\alpha_3> b^{*}+c^{*}$$ 

 and this a contradiction.
\end{enumerate}
\medskip
\medskip

The proof of the case $\;c=p(a+b)$ follows the same steps and it was completely done in \cite{AdrianoMarzullo}. 
\end{proof}
\vspace{3mm}
\newpage
\section{Example $1$}
\vspace{2mm}

 Let $\;\mathbf{a}=(a,b,c)=(a,b,a+b)=(2,3,5)$.  The leading member of the family is:
$$ \mathbf{\underline{a}}=\left(1,\;1+a,\;1+a+b,\;1+a+b+c\right)=(1,\;3,\;6,\;11)$$
 and:
$$ \mathbf{\underline{a}_j}=(j,\;3+j,\;6+j,\;11+j)\;\;\;\mbox{where}\;\;\;j\in \N_{\geq 1}.$$

\vspace{1mm}
  For each $j$, we compute the ranks of the free modules in the minimal free resolution of $P_{\mathbf{\underline{a}_j}}$. 
  Then by analyzing the results obtaining by running the Macaulay $2$ computer program we get the period $T$ starts at $j=22$ and $T=a+b+c=2(a+b)=10$:
\vspace{5mm}
 
\begin{table}[h!]
\caption{Example $1$} 
	\centering
		\begin{tabular}{|c|c|c|}
		\hline
		\vdots                                          &		                                                  &\\
$j=22\;\rightarrow\;(1,\;\mathbf{6},\; 9,\;4,\;0) $ & $j=32\;\rightarrow\;(1,\;\mathbf{6},\; 9,\;4,\;0) $ & $j=42\;\rightarrow\;(1,\;\mathbf{6},\; 9,\;4,\;0)$\\
$j=23\;\rightarrow\;(1,\;\mathbf{4},\; 5,\;2,\;0) $ & $j=33\;\rightarrow\;(1,\;\mathbf{4},\; 5,\;2,\;0) $ & $j=43\;\rightarrow\;(1,\;\mathbf{4},\; 5,\;2,\;0)$\\
$j=24\;\rightarrow\;(1,\;\mathbf{4},\; 5,\;2,\;0) $ & $j=34\;\rightarrow\;(1,\;\mathbf{4},\; 5,\;2,\;0) $ & $j=44\;\rightarrow\;(1,\;\mathbf{4},\; 5,\;2,\;0)$\\
$j=25\;\rightarrow\;(1,\;\mathbf{4},\; 5,\;2,\;0) $ & $j=35\;\rightarrow\;(1,\;\mathbf{4},\; 5,\;2,\;0) $ & $j=45\;\rightarrow\;(1,\;\mathbf{4},\; 5,\;2,\;0)$\\
$j=26\;\rightarrow\;(1,\;\mathbf{6},\; 8,\;3,\;0) $ & $j=36\;\rightarrow\;(1,\;\mathbf{6},\; 9,\;4,\;0) $ & $j=46\;\rightarrow\;(1,\;\mathbf{6},\; 9,\;4,\;0)$\\
$j=27\;\rightarrow\;(1,\;\mathbf{4},\; 5,\;2,\;0) $ & $j=37\;\rightarrow\;(1,\;\mathbf{4},\; 5,\;2,\;0) $ & $j=47\;\rightarrow\;(1,\;\mathbf{4},\; 5,\;2,\;0)$\\
$j=28\;\rightarrow\;(1,\;\mathbf{6},\; 9,\;4,\;0) $ & $j=38\;\rightarrow\;(1,\;\mathbf{6},\; 9,\;4,\;0) $ & $j=48\;\rightarrow\;(1,\;\mathbf{6},\; 9,\;4,\;0)$\\
$j=29\;\rightarrow\;(1,\;\mathbf{3},\; 3,\;1,\;0) $ & $j=39\;\rightarrow\;(1,\;\mathbf{3},\; 3,\;1,\;0) $ & $j=49\;\rightarrow\;(1,\;\mathbf{3},\; 3,\;1,\;0)$\\
$j=30\;\rightarrow\;(1,\;\mathbf{6},\; 9,\;4,\;0) $ & $j=40\;\rightarrow\;(1,\;\mathbf{6},\; 9,\;4,\;0) $ & $j=50\;\rightarrow\;(1,\;\mathbf{6},\; 9,\;4,\;0)$\\
$j=31\;\rightarrow\;(1,\;\mathbf{4},\; 5,\;2,\;0) $ & $j=41\;\rightarrow\;(1,\;\mathbf{4},\; 5,\;2,\;0) $ & $j=51\;\rightarrow\;(1,\;\mathbf{4},\; 5,\;2,\;0)$\\
                                                    &                                                     & \vdots \\ 
\hline
		\end{tabular}
 \end{table}
 \vspace{3mm}
 
 Note that the complete intersection happens at $j=29$, $j=29+T=29+10=39$ and $j=29+2T=29+20=49$.
\vspace{1mm}

\newpage
\section{Example $2$} 
\vspace{2mm}

Let $\mathbf{a}=(a,b,c)=(3(b+c),b,c)=(12,3,1)$. The leading member of the family is:
$$ \mathbf{\underline{a}}=\left(1,\;1+a,\;1+a+b,\;1+a+b+c\right)=(1,\;13,\;16,\;17)$$
 and:
$$ \mathbf{\underline{a}_j}=(j,\;13+j,\;16+j,\;17+j)\;\;\;\mbox{where}\;\;\;j\in \N_{\geq 1}.$$
\vspace{1mm}

 For each $j$, we compute the ranks of the free modules in the minimal free resolution of $P_{\mathbf{\underline{a}_j}}$.
 Then by analyzing the results obtaining by running the Macaulay $2$ computer program we get the period $T$ starts at $j=65$ and $T=a+b+c=4(b+c)=16$:
\vspace{5mm}

 \begin{table}[h!]
\caption{Example $2$} 
	\centering
		\begin{tabular}{|c|c|c|}
		\hline
		\vdots                                          &		                                                  &\\
$j=65\;\rightarrow\;(1,\;\mathbf{5},\; 5,\;1,\;0) $ & $j=81\;\rightarrow\;(1,\;\mathbf{5},\; 6,\;2,\;0) $ & $j=97\;\rightarrow\;(1,\;\mathbf{5},\; 6,\;2,\;0)$\\
$j=66\;\rightarrow\;(1,\;\mathbf{5},\; 6,\;2,\;0) $ & $j=82\;\rightarrow\;(1,\;\mathbf{5},\; 6,\;2,\;0) $ & $j=98\;\rightarrow\;(1,\;\mathbf{5},\; 6,\;2,\;0)$\\
$j=67\;\rightarrow\;(1,\;\mathbf{4},\; 5,\;2,\;0) $ & $j=83\;\rightarrow\;(1,\;\mathbf{4},\; 5,\;2,\;0) $ & $j=99\;\rightarrow\;(1,\;\mathbf{4},\; 5,\;2,\;0)$\\
$j=68\;\rightarrow\;(1,\;\mathbf{5},\; 6,\;2,\;0) $ & $j=84\;\rightarrow\;(1,\;\mathbf{5},\; 7,\;3,\;0) $ & $j=100\;\rightarrow\;(1,\;\mathbf{5},\; 7,\;3,\;0)$\\
$j=69\;\rightarrow\;(1,\;\mathbf{5},\; 7,\;3,\;0) $ & $j=85\;\rightarrow\;(1,\;\mathbf{5},\; 7,\;3,\;0) $ & $j=101\;\rightarrow\;(1,\;\mathbf{5},\; 7,\;3,\;0)$\\
$j=70\;\rightarrow\;(1,\;\mathbf{5},\; 7,\;3,\;0) $ & $j=86\;\rightarrow\;(1,\;\mathbf{5},\; 7,\;3,\;0) $ & $j=102\;\rightarrow\;(1,\;\mathbf{5},\; 7,\;3,\;0)$\\
$j=71\;\rightarrow\;(1,\;\mathbf{4},\; 5,\;2,\;0) $ & $j=87\;\rightarrow\;(1,\;\mathbf{4},\; 5,\;2,\;0) $ & $j=103\;\rightarrow\;(1,\;\mathbf{4},\; 5,\;2,\;0)$\\
$j=72\;\rightarrow\;(1,\;\mathbf{5},\; 7,\;3,\;0) $ & $j=88\;\rightarrow\;(1,\;\mathbf{5},\; 7,\;3,\;0) $ & $j=104\;\rightarrow\;(1,\;\mathbf{5},\; 7,\;3,\;0)$\\
$j=73\;\rightarrow\;(1,\;\mathbf{5},\; 7,\;3,\;0) $ & $j=89\;\rightarrow\;(1,\;\mathbf{5},\; 7,\;3,\;0) $ & $j=105\;\rightarrow\;(1,\;\mathbf{5},\; 7,\;3,\;0)$\\
$j=74\;\rightarrow\;(1,\;\mathbf{5},\; 7,\;3,\;0) $ & $j=90\;\rightarrow\;(1,\;\mathbf{5},\; 7,\;3,\;0) $ & $j=106\;\rightarrow\;(1,\;\mathbf{5},\; 7,\;3,\;0)$\\
$j=75\;\rightarrow\;(1,\;\mathbf{4},\; 5,\;2,\;0) $ & $j=91\;\rightarrow\;(1,\;\mathbf{4},\; 5,\;2,\;0) $ & $j=107\;\rightarrow\;(1,\;\mathbf{4},\; 5,\;2,\;0)$\\
$j=76\;\rightarrow\;(1,\;\mathbf{4},\; 6,\;3,\;0) $ & $j=92\;\rightarrow\;(1,\;\mathbf{4},\; 6,\;3,\;0) $ & $j=108\;\rightarrow\;(1,\;\mathbf{4},\; 6,\;3,\;0)$\\
$j=77\;\rightarrow\;(1,\;\mathbf{4},\; 6,\;3,\;0) $ & $j=93\;\rightarrow\;(1,\;\mathbf{4},\; 6,\;3,\;0) $ & $j=109\;\rightarrow\;(1,\;\mathbf{4},\; 6,\;3,\;0)$\\
$j=78\;\rightarrow\;(1,\;\mathbf{4},\; 6,\;3,\;0) $ & $j=94\;\rightarrow\;(1,\;\mathbf{4},\; 6,\;3,\;0) $ & $j=110\;\rightarrow\;(1,\;\mathbf{4},\; 6,\;3,\;0)$\\
$j=79\;\rightarrow\;(1,\;\mathbf{3},\; 3,\;1,\;0) $ & $j=95\;\rightarrow\;(1,\;\mathbf{3},\; 3,\;1,\;0) $ & $j=111\;\rightarrow\;(1,\;\mathbf{3},\; 3,\;1,\;0)$\\
$j=80\;\rightarrow\;(1,\;\mathbf{5},\; 5,\;1,\;0) $ & $j=96\;\rightarrow\;(1,\;\mathbf{5},\; 6,\;2,\;0) $ & $j=112\;\rightarrow\;(1,\;\mathbf{5},\; 6,\;2,\;0)$\\
                                                    &                                                     & \vdots \\ 
\hline
		\end{tabular}
 \end{table}
 \vspace{3mm}
 
 Note that the complete intersection happens at $j=79$, $j=79+T=79+16=95$ and $j=79+2T=29+32=111$.
 
\vspace{3mm}
\newpage
\section{Example $3$}
\vspace{2mm}

Let  $\mathbf{a}=(a,b,c)=(a,a+c,c)=(3,5,2)$. The leading member of the family is:
$$ \mathbf{\underline{a}}=\left(1,\;1+a,\;1+a+b,\;1+a+b+c\right)=(1,\;4,\;9,\;11)$$
 and:
$$ \mathbf{\underline{a}_j}=(j,\;4+j,\;9+j,\;11+j)\;\;\;\mbox{where}\;\;\;j\in \N_{\geq 1}.$$
\vspace{1mm}

For each $j$, we compute the ranks of the free modules in the minimal free resolution of $P_{\mathbf{\underline{a}_j}}$.
  Then by analyzing the results obtaining by running the Macaulay $2$ computer program we get the period $T$ starts at $j=32$ and $T=a+b+c=10$:
\vspace{5mm}
 
\begin{table}[h!]
\caption{Example $3$} 
	\centering
		\begin{tabular}{|c|c|c|}
		\hline
		\vdots                                           &		                                                  &\\
$j=32\;\rightarrow\;(1,\;\mathbf{10},\; 16,\;7,\;0)$ & $j=42\;\rightarrow\;(1,\;\mathbf{10},\; 16,\;7,\;0) $ & $j=52\;\rightarrow\;(1,\;\mathbf{10},\; 16,\;7,\;0)$\\
$j=33\;\rightarrow\;(1,\;\;\mathbf{8},\; 12,\;5,\;0)$  & $j=43\;\rightarrow\;\;(1,\;\mathbf{8},\; 12,\;5,\;0) $ & $j=53\;\rightarrow\;\;(1,\;\mathbf{8},\; 12,\;5,\;0)$\\
$j=34\;\rightarrow\;(1,\;\mathbf{10},\; 16,\;7,\;0) $ & $j=44\;\rightarrow\;(1,\;\mathbf{10},\; 17,\;8,\;0) $ & $j=54\;\rightarrow\;(1,\;\mathbf{10},\; 17,\;8,\;0)$\\
$j=35\;\rightarrow\;(1,\;\;\mathbf{9},\; 14,\;6,\;0) $ & $j=45\;\rightarrow\;\;(1,\;\mathbf{9},\; 14,\;6,\;0) $ & $j=55\;\rightarrow\;\;(1,\;\mathbf{9},\; 14,\;6,\;0)$\\
$j=36\;\rightarrow\;(1,\;\;\mathbf{9},\; 15,\;7,\;0) $ & $j=46\;\rightarrow\;\;(1,\;\mathbf{9},\; 15,\;7,\;0) $ & $j=56\;\rightarrow\;\;(1,\;\mathbf{9},\; 15,\;7,\;0)$\\
$j=37\;\rightarrow\;(1,\;\;\mathbf{8},\; 12,\;5,\;0) $ & $j=47\;\rightarrow\;\;(1,\;\mathbf{8},\; 12,\;5,\;0) $ & $j=57\;\rightarrow\;\;(1,\;\mathbf{8},\; 12,\;5,\;0)$\\
$j=38\;\rightarrow\;(1,\;\mathbf{10},\; 17,\;8,\;0) $ & $j=48\;\rightarrow\;(1,\;\mathbf{10},\; 17,\;8,\;0) $ & $j=58\;\rightarrow\;(1,\;\mathbf{10},\; 17,\;8,\;0)$\\
$j=39\;\rightarrow\;(1,\;\;\mathbf{8},\; 12,\;5,\;0) $ & $j=49\;\;\rightarrow\;(1,\;\mathbf{8},\; 12,\;5,\;0) $ & $j=59\;\rightarrow\;\;(1,\;\mathbf{8},\; 12,\;5,\;0)$\\
$j=40\;\rightarrow\;(1,\;\;\mathbf{9},\; 15,\;7,\;0) $ & $j=50\;\rightarrow\;\;(1,\;\mathbf{9},\; 15,\;7,\;0) $ & $j=60\;\rightarrow\;\;(1,\;\mathbf{9},\; 15,\;7,\;0)$\\
$j=41\;\rightarrow\;(1,\;\;\mathbf{9},\; 14,\;6,\;0) $ & $j=51\;\rightarrow\;\;(1,\;\mathbf{9},\; 14,\;6,\;0) $ & $j=61\;\rightarrow\;\;(1,\;\mathbf{9},\; 14,\;6,\;0)$\\
                                                    &                                                     & \vdots \\ 
\hline
		\end{tabular}
 \end{table}

Note that, in this case, the hypotheses of Theorem $1$ fail and so its conclusion is false. 

\vspace{5mm}

\begin{oss}In general,  it looks like that for $\;(a,b,c)=(a,\;p(b+c),\;c)$ we do not have a complete intersection ideal in the period.
\end{oss}
\vspace{3mm}

It would be interesting to prove the Herzog-Srinivasan conjecture for any given number $a,b,c$ in the case $m=4$, and it would be fascinating to analyze what happens for $m>4$.

\end{document}